\documentclass[12pt,a4paper]{article}
\usepackage{t1enc}
\usepackage[utf8]{inputenc}
\usepackage[english]{babel}
\usepackage{amsmath}
\usepackage{amssymb}
\usepackage{amsthm}
\usepackage{hyperref}
\usepackage{color}

\newcommand{\C}{\mathbb{C}}
\newcommand{\p}{\mathbb{P}}

\newcommand{\F}{\mathbb{F}}

\newcommand{\lra}{\longrightarrow}

\newcommand{\Lra}{\Longrightarrow}

\newcommand{\ff}{\mathcal{F}}

\newcommand{\zz}{\mathbb{Z}}

\newcommand{\D}{\mathcal D}

\newcommand{\kk}{\mathbb K}
\newcommand{\xx}{\mathcal X}
\newcommand{\yy}{\mathcal Y}

\newcommand{\cc}{\mathcal C}

\newcommand{\Char}{\operatorname{char}}

\newcommand{\aut}{\operatorname{Aut}}

\newcommand{\E}{\mathbb{E}}
\newcommand{\Ei}{\mathbb{E}_i}

\theoremstyle{plain}
\newtheorem{thm}{Theorem}[section]

\newtheorem{prop}[thm]{Proposition}
\newtheorem{lem}[thm]{Lemma}
\newtheorem{cor}[thm]{Corollary}
\newtheorem{rem}[thm]{Remark}

\title{Algebraic curves admitting automorphism groups of large prime square order}

\author{Nazar Arakelian and Diego Kian}

\begin{document}

\maketitle
\begin{abstract}
Let $\kk$ denote an algebraically closed field of arbitrary characteristic. In this paper we provide bounds for the size of a prime $\ell$ for which there are curves defined over $\kk$ admitting automorphism groups of order $\ell^2$. In addition, we also give a classification of the families of curves attaining the upper bounds for $\ell$ in both tame and wild case. Finally, we present the full automorphism groups of the curves attaining the highest bounds.
\end{abstract}


\section {Introduction}\label{intro}

Let $\xx$ be a (projective, algebraic, nonsingular, absolutely irreducible) curve of genus $g$ defined over an algebraically closed field $\kk$ of characteristic $p \geq 0$.  If $\kk(\xx)$ stands for the function field of $\xx$ over $\kk$, the full automorphism group of $\xx$, denoted by $\aut(\xx)$, consists of the group of the field automorphisms of $\kk(\xx)$ fixing $\kk$ elementwise with the composition operation. By \cite[Lemma 11.22]{HKT}, if $\xx \subset \p^n$ is nonsingular, then $\aut(\xx)$ acts on the points of $\xx$ as a linear collineation subgroup of $\operatorname{PGL}(n,\kk)$.

It is well known that $\aut(\xx)$ is infinite when $g=0$ and $g=1$. In contrast, if $g \geq 2$, then $\aut(\xx)$ is finite. More precisely, if $G \leq \aut(\xx)$ the Hurwitz bound states that
\begin{equation}\label{hb}
|G| \leq 84(g-1)
\end{equation}
with some exceptions occurring only when $\Char(\kk)>0$, see \cite[Theorem 11.56]{HKT}. As a matter of fact, one has $|G| < 16g^4$ up to one exception, which is the Hermitian curve defined by $y^q+y=x^{q+1}$, where $q=p^h$ for some integer $h \geq 1$ and $p=\Char(\kk)$. Over the last few decades these bounds were improved for certain types of groups or for curves over specific ground fields, specially for fields of positive characteristic, where the bound \eqref{hb} fails. For instance, if $G$ is an abelian group, we have that $|G|\leq 4g+4$, see e.g. \cite[Theorem 11.79]{HKT}. For cyclic groups, interesting results were obtained in \cite{IS}: curves defined over the complex numbers $\C$ admitting an automorphism of order $N >2g+1$ where classified. In \cite{DGT} this result was specialized for curves over a general algebraically closed field $\kk$. In the particular case where $\xx$ admits automorphism of prime order $\ell$, it follows by \cite{Ho} that either $\ell=2g+1$ or $\ell\leq g+1$; in this same reference, curves with an automorphism of order $2g+1$ are classified, and in \cite{AS} a classification  is provided for curves with an automorphism of order $g+1$. 

Other important results in positive characteristic are known for curves having solvable automorphism groups \cite{KM2,MS}, automorphism groups of prime power order \cite{Na,KM},  among other. It is worthwhile to mention that the classification of algebraic curves defined over fields $\kk$ of zero characteristic via automorphism groups are also of great importance. This can be seen in the particular case where $\kk=\C$ since algebraic curves over the complex field correspond to Riemann surfaces; in this direction, one can mention \cite{RS,Zo} and the references within.

In this paper we deal with algebraic curves $\xx$ of genus $g \geq 2$ defined over an algebraically closed field $\kk$ of arbitrary characteristic admitting automorphism groups $G$ of order $\ell^2$, where $\ell$ is a prime number. Specifically, we will provide bounds for the size of such a prime $\ell$ (Theorem \ref{main1}). As we will see, the upper bound in this specific case improves previous known bounds of \cite{Ho,Na,KM}.  Furthermore, we also give a classification of the families of curves attaining the upper bounds for $\ell$ in both tame and wild cases (Theorem \ref{main2}). The only curve attaining the upper bound for $\ell$ in the tame case is the Fermat curve $x^\ell+y^\ell=1$, and its full automorphism group is well known. We will finish the paper by presenting the full automorphism groups of the curves attaining the upper bound in the wild case as well (Theorem \ref{main3}).

Summarized, the main results of the paper are the following. 

\begin{thm}\label{main1}
Let $\xx$ be an irreducible algebraic curve of genus $g \geq 2$ defined over an algebraically closed field $\kk$ of characteristic $p \geq 0$.  Assume that there exists a subgroup $G\leq \aut(\xx)$ such that $|G|=\ell^2$, where $\ell$ is a prime number. Then 
$$
\ell \leq \frac{3+\sqrt{8g+1}}{2}.
$$
More precisely, we have:
\begin{itemize}
\item[(a)] If $\ell \neq p$, then either $\ell  = \frac{3+\sqrt{8g+1}}{2}$, or $\ell  = \frac{1+\sqrt{8g+1}}{2}$,  or 
$$
\ell  \leq \sqrt{2g+1}.
$$
\item[(b)] If $\ell=p>0$, then either $p=\frac{1+\sqrt{8g+1}}{2}$, or $p=\sqrt{2g+1}$ or
$$
p \leq \sqrt{g}+1.
$$
\end{itemize}
\end{thm}

The next result characterizes the curves attaining the bounds of Theorem \ref{main1}.

\begin{thm}\label{main2}
Let $\xx$ be an irreducible algebraic curve of genus $g \geq 2$ defined over an algebraically closed field $\kk$ of characteristic $p \geq 0$.  Assume that there exists a subgroup $G\leq \aut(\xx)$ such that $|G|=\ell^2$, where $\ell$ is a prime number. Then
\begin{itemize}
\item[(a)] Suppose that $\ell\neq p$.
\begin{itemize}
\item[(a.1)] $\ell  = \frac{3+\sqrt{8g+1}}{2}$ if and only if $\xx$ is birationally equivalent to the Fermat curve $$\ff: x^\ell+y^\ell=1.$$
\item[(a.2)] $\ell  = \frac{1+\sqrt{8g+1}}{2}$ if and only if  $\xx$ is birationally equivalent to a curve of the family  
$$
\mathcal{K}^{1}_{\alpha,\beta}: y^{\ell^2}=x^\alpha(x-1)^{\ell \beta},
$$
where $\alpha$ and $\beta$ are positive integers such that $\alpha +\beta \ell <\ell^2$ with $\ell \nmid \alpha$.
\item[(a.3)] $\ell=\sqrt{2g+1}$ if, and only if, $\xx$ is either birationally equivalent to a curve of the family 
$$
\mathcal{K}^{2}_{\alpha,\beta}: y^{\ell^2}=x^\alpha(x-1)^{\beta},
$$
where $\alpha$ and $\beta$ are positive integers such that $\alpha +\beta <\ell^2$ with $\alpha$, $\beta$ and $\alpha+\beta$ co-prime with $\ell$, or birationally equivalent to a curve of the family
$$
\mathcal{G}_a:ax^3y^3+x^3+y^3=1,
$$
with $a\in \kk^* $ or birationally equivalent to a curve of the family
$$
\mathcal{N}_a:(y^3-x^3)^2+(a-1)(x^3-ay^3)=0, 
$$
with $a \in \kk \backslash \{0,1\}$.
In the latter two cases, $\ell=3$ and $g=4$.
\end{itemize}
\item[(b)] Suppose that $\ell=p>0$.
\begin{itemize}
\item[(b.1)]  $p=\frac{1+\sqrt{8g+1}}{2}$ if and only if $p>2$ and $\xx$ is birationally equivalent to a curve of the family
$$
\yy_{a}:y^p-y=x^{p+1}+ax^2, 
$$
with $a \in \kk$.
\item[(b.2)]  $p=\sqrt{2g+1}$ if and only if $p>2$ and $\xx$ is birationally equivalent to a curve of the family
$$
\cc_{a,b}: \big((x^p-x)-a(y^p-y)\big)^2=b^2(y^p-y),
$$
where $a,b \in \kk$, with $a \notin \mathbb{F}_p$ and $b \neq 0$.
\item[(b.3)]  $p=\sqrt{g}+1$ if and only if $p>2$ and $\xx$ is birationally equivalent to a curve of the Artin-Mumford type
$$
\mathcal{M}_c: (y^p-y)(x^p-x)=c,
$$
with $c\in \kk^*$.
\end{itemize}
\end{itemize}
\end{thm}

The full automorphism group of the Fermat curve $\ff$ is well known: If $n=p^r+1$ for some $r>0$, where $p=\Char(\kk)$, then $\aut(\ff) \cong \operatorname{PGU}(3,p^r)$ (see e.g. \cite[Proposition 11.30]{HKT}); otherwise $\aut(\ff)$ has a normal subgroup $H \cong \zz_\ell \times \zz_\ell$, and $\aut(\ff)/H \cong S_3$; in particular, $|\aut(\ff)|=6\ell^2$, see \cite[Theorem 11.31]{HKT}. In the case of the Artin-Mumford curve $\mathcal{M}_c$, we have from \cite[Theorem 7]{VM} that $\aut(\mathcal{M}_c) \cong (C_p \times C_p) \rtimes D_{p-1}$, where $C_p$ denotes a cyclic group of order $p$ and $D_{p-1}$ denotes a dihedral group of order $2(p-1)$. In turn, note that the curve $\yy_{0}$ is isomorphic to the Hermitian curve. Thus $\aut(\yy_{0}) \cong \operatorname{PGU}(3,p)$, see e.g. \cite{Le}. In general, concerning the curve $\yy_{a}$ we will prove the following result.

\begin{thm}\label{main3}
The curve $\yy_{a}:y^p-y=x^{p+1}+ax^2$ is birationally equivalent to the Hermitian curve $\mathcal{H}:u^{p+1}=v^p+v$ if and only if $a=0$. Furthermore, if $a\neq 0$ then $\aut(\yy_{a})$ fixes the only point at the infinity of $\yy_{a}$. More precisely, $\aut(\yy_{a})$ contains a unique normal Sylow $p$-subgroup $S_p$ which is isomorphic to the Heisenberg group of order $|S_p|=p^3$, and 
$$
\frac{\aut(\yy_{a})}{S_p} \cong \zz_{p-1}.
$$
In particular, $|\aut(\yy_{a})|=p^3(p-1)=4g^2+g(1+\sqrt{8g+1})$.
\end{thm}

\begin{rem}
In \cite[Example 3.7]{BNZ} it is mentioned that the curve defined by $y^q-y=x(x^q-x)$, where $q$ is a power of $p=\Char(\kk)$ is well known, and it is known that its automorphism group has order $q^3(q-1)$. It is possible that the proof of Theorem \ref{main3} could also be derived from the proof of this result. However, we did not find the proof of such statement in the literature. 
\end{rem}

We say that a curve $\xx$ of genus $g$ has large automorphism group if $|\aut(\xx)|$ exceeds the Hurwitz bound, that is, if $|\aut(\xx)|>84(g-1)$. From Theorem \ref{main3} the curves $\yy_{a,b}$ has always a large automorphism group for a sufficiently large value of $g$.

When $\Char(\kk)=p>0$, we have associated to a curve $\xx$ defined over $\kk$ the so called $p$-rank of $\xx$, which is denoted by $\gamma(\xx)$, and it is defined as the dimension of the $\F_p$-vector space of the $p$-torsion points of the Jacobian variety $\operatorname{Jac}(\xx)$ of the curve $\xx$. In \cite[Theorem 1]{Na}, Nakajima provided a bound for the order of the Sylow $p$-subgroups of curves $\xx$ defined over $\kk$ that depends on the $p$-rank of $\xx$. A curve attaining such bound is called a Nakajima extremal curve. The next result is a bi-product of the proof of Theorem \ref{main3}.

\begin{cor}\label{corm}
For all $a,b \in \kk$, the curve $\yy_{a}:y^p-y=x^{p+1}+ax^2$ has zero $p$-rank and it is Nakajima extremal.
\end{cor}

\begin{rem}
Using techniques from the theory of Riemann surfaces, Zomorrodian establishes in \cite[Theorem 1.1.2]{Zo} a bound for the size of $\ell$-Sylow subgroups $S_\ell$ of $\aut(\xx)$  where $\xx$ is a compact Riemann surface of genus $g \geq 2$, namely $|S_2| \leq 16(g-1)$, $|S_3| \leq 9(g-1)$ and $|S_\ell| \leq \frac{2\ell(g-1)}{\ell-3}$ when $\ell >3$. It should be noted that such a bound for $\ell>3$ coincides with the one given in Theorem \ref{main1}(a) for $\kk=\C$ .
\end{rem}

\section{Background}\label{back}

Our notation and terminology are standard. For an exhaustive treatise of the theory of curves and algebraic function fields, the reader is referred to \cite{HKT} and \cite{St}. Let $\xx$ be a (projective, geometrically irreducible, algebraic) curve defined over an algebraically closed field $\kk$ of characteristic $p \geq 0$.  We denote by $\kk(\xx)$ the function field of $\xx$. By a point $P \in \xx$ we mean a point in a non-singular model of $\xx$; in this way, we have a one-to-one correspondence between points of $\xx$ and places of $\kk(\xx)$. 

Let $\aut(\xx)$ denote the full automorphism group of $\xx$. For a subgroup $G$ of $\aut(\xx)$, we denote by $\kk(\xx)^G$ the fixed field of $G$. A non-singular model $\bar{\xx}$ of  $\kk(\xx)^G$ is referred to as the quotient curve of $\xx$ by $G$ and denoted by $\xx/G$. The field extension $\kk(\xx):\kk(\xx)^G$ is Galois with Galois group $G$. For a point $P \in \xx$, $G(P)$ stands for the orbit of $P$ under the action of $G$ on $\xx$ seen as a point-set. The orbit $G(P)$ is said to be long if $|G(P)| = |G|$, short otherwise. There is a one-to-one correspondence between short orbits and ramified points in the extension $\kk(\xx):\kk(\xx)^G$. The group $G$ might have no short orbits; if this is the case, the cover $\xx \rightarrow \xx/G$ (or equivalently, the extension $\kk(\xx):\kk(\xx)^G$) is unramified. 

For $P \in \xx$, the subgroup $G_P$ of $G$ consisting of all elements of $G$ fixing $P$ is called the stabilizer of $P$ in $G$. For a non-negative integer $i$, the $i$-th ramification group of $\xx$ at $P$ is denoted by $G_P^{(i)}$, and defined by
$$
G_P^{(i)}=\{\sigma \in G_P \ | \ v_P(\sigma(t)-t)\geq i+1\}, 
$$
 where $t$ is a local parameter at $P$ and $v_P$ is the respective discrete valuation. Here we have a filtration
 $$
 G_P^{(0)} \geq G_P^{(1)} \geq \cdots \geq G_P^{(i)} \geq \cdots,
 $$
 and $G_P^{(j)}$ is trivial for some $j\geq 0$. Moreover, $G_P=G_P^{(0)}$, $G_P^{(1)}$ is the unique Sylow $p$-subgroup of $G_P^{(0)}$, and the factor group $G_P^{(0)}/G_P^{(1)}$ is cyclic of order prime to $p$, see \cite[Theorem 11.74]{HKT}. In particular, if $p=0$ then $G_P^{(0)}$ is cyclic and $G_P^{(j)}$ is trivial for $j>0$, and when $p>0$ and $|G_P|$ is a power of  $p$, then $G_P=G_P^{(0)}=G_P^{(1)}$. For a point $P \in \xx$, the ramification index of $P$ is defined as   $e_P := |G^{(0)}_P|$ and the different exponent of $P$ is 
 $$
 d_P := \sum_{i = 0}^{\infty}(|G_P^{(i)}|- 1).
 $$ 
 
 Let $g$ and $\bar{g}$ be the genus of $\xx$ and $\bar{\xx}=\xx/G$, respectively. The Riemann-Hurwitz genus formula is 
\begin{equation}\label{rhg}
2g-2=|G|(2\bar{g}-2)+\sum_{P \in \xx}d_P, 
\end{equation}
see \cite[Theorem 11.72]{HKT}.
 If $\ell_1,\ldots,\ell_k$  are the sizes of the short orbits of $G$, then (\ref{rhg}) yields
\begin{equation}\label{rhso}
2g-2 \geq |G|(2\bar{g}-2)+\sum_{\nu=1}^{k} \big(|G|-\ell_\nu\big),
\end{equation}
and equality holds if $\gcd(|G_P|,p)=1$ for all $P \in \xx$; see \cite[Theorem 11.57 and Remark 11.61]{HKT}. 

Let $\gamma$ be the $p$-rank of $\xx$. Then $0 \leq \gamma \leq g$, and when $\gamma=g$ the curve $\xx$ is \emph{ordinary} or \emph{general}; see \cite[Theorem 6.96]{HKT}.
For a $p$-subgroup $S$ of $\aut(\xx)$, the Deuring-Shafarevich formula is
\begin{equation}\label{dsg}
\gamma-1=|S|(\bar{\gamma}-1)+\sum_{\nu=1}^{k} \big(|S|-\ell_\nu);
\end{equation}
where $\bar{\gamma}$ is the $p$-rank of $\bar{\xx}=\xx/S$ and $\ell_1,\ldots,\ell_k$  are the sizes of the short orbits of $S$; see \cite[Theorem 11.62]{HKT}. 

We now state some results that will be useful in the sequel of the paper. 

\begin{prop}\cite[Lemma 11.75]{HKT}\label{congp}
Assume $p>0$ and let $P \in \xx$. The integers $k >0$ satisfying $G_P^{(k)} \neq G_P^{(k+1)}$ are all congruent modulo $p$.
\end{prop}

\begin{prop}\cite[Theorem 11.78]{HKT}\label{nak}
Assume $p>0$ and let $P \in \xx$, where $\xx$ has positive genus $g$. Let $G_P$ be a $\kk$-automorphism group fixing $P$. Let $F_i$ be the subfield of $\kk(\xx)$ fixed by $G_P^{(i)}$. If $F_1$ is not rational, then
$|G_P^{(1)}| \leq g$.  Furthermore, if $F_1$ is rational and $G_P^{(1)}$ has a short orbit other than $\{P\}$, then
$$
|G_P^{(1)}| \leq \frac{p}{p-1}g.
$$
\end{prop}

The following result is stated and proved  in \cite[Theorem 27]{DGT} for an algebraically closed field of characteristic $p>2$ to fit the context of the paper.  However, the proof works for curves over arbitrary algebraically closed fields.

\begin{thm}\cite[Theorem 27]{DGT}\label{dgt}
Let $\xx$ be a curve of genus $g\geq 2$ defined over an algebraically closed field $\kk$ of characteristic $p \geq 0$ admitting an automorphism of order $N \geq 2g+1$, where $p \nmid N$. Then, up to birational equivalence, either
$$
\xx: y^N=x^\alpha(x-1)^\beta
$$
where $\alpha$ and $\beta$ are positive integers such that $\alpha +\beta <N$ with $\gcd(\alpha,\beta,N)=1$, or
$$
\xx:y^2=(x^{N/2}-1)(x^{N/2}-\lambda)
$$
where $\lambda \in \kk \backslash \{0,1\}$ and $N \equiv 2 \mod 4$.
\end{thm}

\begin{thm}\cite[Theorem 1]{Ch}\label{chang}
Let $\cc$ and $\D$ two nonsingular plane curves over an algebraically closed field with projective degree $d \geq 4$. Then $\cc$ and $\D$ are birationally equivalent if and only if they are projectively equivalent.
\end{thm}

\section{Proof of Theorem \ref{main1}}\label{pm1}

In this section we provide a proof of Theorem \ref{main1}. From now on, $\xx$ stands for a curve of genus $g\geq 2$ defined over an algebraically closed field $\kk$ of characteristic $p \geq 0$. We are assuming that the full automorphism group $\aut(\xx)$ of $\xx$ admits a subgroup $G$ of order $\ell^2$, where $\ell$ is a prime integer. It is important to recall that in such case the group $G$ is abelian. We will study the tame case ($\ell \neq p$) and the wild case ($\ell =p$) separately. 

\subsection{Tame case}\label{tc}

For this subsection, suppose $\ell \neq p=\Char(\kk)$. Since $\ell$ is prime and $|G|=\ell^2$, the short orbits of $G$ on $\xx$ have either size $1$ or size $\ell$. Thus the Riemann-Hurwitz formula \eqref{rhso} in this case reads
\begin{equation}\label{rhtame}
2g-2  = \ell^2(2\bar{g}-2)+r(\ell^2-1)+s(\ell^2-\ell),
\end{equation}
where $\bar{g}$ denotes the genus of $\xx/G$, $r$ is the number of orbits of $G$ of size $1$ (i.e., the number of fixed points of $G$) and $s$ is the number of orbits of $G$ of size $\ell$. The proof of Theorem \ref{main1}(a) will be a direct consequence of the following case by case analysis:
\begin{itemize}
\item[(i)] \emph{Case $\bar{g} \geq 2$:} Since $r,s \geq 0$ by \eqref{rhtame} we have $2g-2 \geq 2\ell^2(\bar{g}-1)\geq 2\ell^2$, hence $\ell \leq \sqrt{g-1} < \sqrt{2g+1}$.
\item[(ii)] \emph{Case $\bar{g} =1$:} From \eqref{rhtame} we have that the case $r=s=0$ is impossible. So assume first that $r \geq 1$. From \eqref{rhtame} we immediately obtain $\ell \leq \sqrt{2g-1}<\sqrt{2g+1}$. Now let us suppose that $r=0$ and $s\geq 1$. If $s\geq 2$, we obtain from \eqref{rhtame} that $\ell^2-\ell-(g-1)\leq 0$. Taking into account that $\ell>1$, the solution of the last inequality is 
$$
\ell \leq \frac{1+\sqrt{4g-3}}{2}<\sqrt{2g+1}.
$$
Finally, suppose $r=0$ and $s=1$. We then obtain $\ell^2-\ell=2g-2$, and 
$$
\ell = \frac{1+\sqrt{8g-7}}{2}.
$$
 Note that if $\ell>2$, then $\ell^2=2g+(\ell-2)\geq 2g+1$. By \cite[Proposition 10]{DGT} if $G$ is cyclic, then the number of branch points of $\xx/G$ with respect to the cover $\xx \lra \xx/G$ is bigger than $1$, contradicting our assumption $r=0$ and $s=1$ . Hence, if $\ell>2$ then  $G$ is not cyclic. From $r=0$ and $s=1$ we conclude that $G$ has a single short orbit $\Omega$ of size $\ell$. Let $P\in \Omega$ and consider the stabilizer $G_P \leq G$ of $P$ in $G$. The order of $G_P$ is $\ell$ and since $G$ is abelian, $G_P$ is the stabilizer of every point $Q \in \Omega$. From the fact that $G \cong \zz_\ell \times \zz_\ell$, we have that there exists a subgroup $H \leq G$ of order $\ell$ other than $G_P$. The cover $\xx \rightarrow \xx/H$ is then unramified, and if $\tilde{g}$ denotes the genus of $\xx/H$, the Riemann-Hurwitz formula applied to such cover gives
 $$
 2g-2=\ell(2\tilde{g}-2)  \Longrightarrow \ell=2\tilde{g}-1.
 $$
 Hence $\xx/H$ is a $2\tilde{g}-1$-curve, where $2\tilde{g}-1$ is a prime. By \cite[Theorem 2]{Ho} we must have $\tilde{g}=2$, thus $\ell=3$. In this case,  $\ell^2-\ell=2g-2 \Longrightarrow g=4$. On the other hand, if $\ell=2$ then $\ell^2-\ell=2g-2 \Longrightarrow g=2$. Since
 $$
\frac{1+\sqrt{8g-7}}{2} \leq \sqrt{2g+1} \Longleftrightarrow g \leq 4,
$$
we obtain $\ell \leq  \sqrt{2g+1}$.
 
\item[(iii)] \emph{Case $\bar{g} =0$:} Again, the case $r=s=0$ clearly does not occur. We have then the following sub-cases, by the systematic use of \eqref{rhtame}.
\begin{itemize}
\item If $r \geq 3$, then \eqref{rhtame} gives $2g-2 \geq -2\ell^2+3\ell^2-3$, whence $\ell \leq \sqrt{2g+1}$. Furthermore, one should note that $r\geq 4 \Lra \ell \leq \sqrt{g+1}$ and $r=3$ and $s\geq 1$ implies $\ell \leq \frac{1+\sqrt{16g+9}}{4}<\sqrt{2g+1}$.
\item Assume $r=2$. Then $g \geq 2$ forces $s \geq 1$. In this case, we obtain $\ell^2-\ell-2g \leq 0$, and the solution for this inequality is 
\begin{equation}\label{L2}
\ell \leq \frac{1+\sqrt{8g+1}}{2}. 
\end{equation}
Note that if $s \geq 2$ then $\ell \leq \frac{1+\sqrt{4g+1}}{2}<\sqrt{2g+1}$. In particular, \eqref{L2} can be attained only if $r=2$ and $s=1$.
\item The case $r=1$ cannot occur. Indeed, since $r=1$ we conclude that $G$ fixes some point $P \in \xx$. Thus $G=G_P=G_P^{(0)}$ is cyclic (see e.g. \cite[Proposition 3.8.5]{St}). Let $H \leq G$ be the unique subgroup of $G$ of order $\ell$. By the uniqueness of $H$, every point on a short orbit of $G$ must be fixed by $H$. Then the cover $\xx/H \lra \xx/G$ has a unique (fully) ramified point, with the remaining points splitting completely, which is impossible by the Riemann-Hurwitz formula applied to this cover.
\item Suppose $r=0$. Then $2g-2=(s-2)\ell^2-s\ell$, which provides $s \geq 3$. If $s=3$, then we obtain
\begin{equation}\label{L1}
\ell = \frac{3+\sqrt{8g+1}}{2}. 
\end{equation}
Moreover, $s \geq 4$ implies $\ell \leq \sqrt{g}+1$. In turn, $\sqrt{g}+1 <\sqrt{2g+1}$ for $g \geq 5$ and  $\lfloor \sqrt{g}+1\rfloor =\lfloor\sqrt{2g+1}\rfloor$ for $g=2,3,4$. In particular, \eqref{L1} can only be attained if $r=0$ and $s=3$.
\end{itemize}
\end{itemize}
The proof of Theorem \ref{main1}(a) is now completed.

\subsection{Wild case}\label{wc}

In this subsection we assume that $|G|=p^2$, where $p=\Char(\kk)$. Since $2<\sqrt{g}+1$ for all $g \geq 2$, we may assume that $p>2$. Thus the formula \eqref{rhso} in this case reads
\begin{equation}\label{rhwild}
2g-2  \geq  p^2(2\bar{g}-2)+r(p^2-1)+s(p^2-p),
\end{equation}
where $\bar{g}$ denotes the genus of $\xx/G$, $r$ is the number of orbits of $G$ of size $1$ and $s$ is the number of orbits of $G$ of size $p$. Once again we will perform a case by case analysis.
\begin{itemize}
\item[(i)] \emph{Case $\bar{g} \geq 2$:} It follows immediately from \eqref{rhwild} that $p \leq \sqrt{g-1}<\sqrt{g}+1$.
\item[(ii)] \emph{Case $\bar{g} = 1$:}  From \eqref{rhg} the case $r=s=0$ can be ruled out. Suppose $r \geq 1$. Then $G$ fixes a point $P \in \xx$, that is, $G=G_P$. Since $|G|=p^2$, we have $G=G_P^{(0)}=G_P^{(1)}$. Since $\xx/G_P^{(1)}$ is not rational, from Proposition \ref{nak} we have $p^2=|G_P^{(1)}|\leq g$, that is, $\ell \leq \sqrt{g}<\sqrt{g}+1$.

Now if $r=0$, then $s\geq 1$. Then $G$ has at least one short orbit $\Omega$ of size $p$, and the stabilizers of such points in $G$ have order $p$. Keeping in mind that $G_P=G_P^{(0)}=G_P^{(1)}$ for all $P \in \xx$, we have from \eqref{rhg}
$$
2g-2=p^2(2\bar{g}-2)+\sum_{P \in \xx}\sum_{i = 0}^{\infty}(|G_P^{(i)}|- 1)\geq \sum_{P \notin \Omega}\sum_{i = 0}^{\infty}(|G_P^{(i)}|- 1)+2p(p-1).
$$
Thus $p^2-p-(g-1) \leq 0$, which gives
$$
p \leq \frac{1+\sqrt{4g-3}}{2}<\sqrt{g}+1.
$$
\item[(iii)] \emph{Case $\bar{g} = 0$:} Again, $r=s=0$ does not occur. 
\begin{itemize}
\item Suppose $r \geq 2$. Then $G$ fixes two points of $\xx$, say $P$ and $Q$. We have $G=G_P^{(0)}=G_P^{(1)}$, the quotient curve $\xx/G_P^{(1)}$ is rational and there is a short orbit of $G$ in $\xx$ other than $\{P\}$. From Proposition \ref{nak}, one has
$$
|G|=|G_P^{(1)}| \leq \frac{p}{p-1}g \Lra p^2-p-g \leq 0.
$$
Therefore, 
$$
p \leq \frac{1+\sqrt{4g+1}}{2}<\sqrt{g}+1.
$$
\item Assume $r=1$ and $s \geq 1$. In this case, we have again that $G=G_P^{(0)}=G_P^{(1)}$ for some $P \in \xx$, the quotient curve $\xx/G_P^{(1)}$ is rational and there is a short orbit of $G$ in $\xx$ other than $\{P\}$. The conclusion $p <\sqrt{g}+1$ follows as in the case $r\geq 2$.
\item Assume $r=1$ and $s=0$. Here, we have that $G=G_P^{(0)}=G_P^{(1)}$ for some $P \in \xx$. Moreover, one of the following holds: $|G_P^{(2)}|=1$, $|G_P^{(2)}|=p$ of $|G_P^{(2)}|=p^2$.  Let us first show that the sub-case $|G_P^{(2)}|=1$ does not occur. Indeed, in this case \eqref{rhg} provides
$$
2g-2=-2p^2+2(p^2-1) \Lra g=0,
$$
a contradiction.

Now suppose that $|G_P^{(2)}|=p$. By Proposition \ref{congp} we must have 
$$
|G_P^{(2)}|=|G_P^{(3)}|=\cdots =|G_P^{(p+1)}|=p.
$$
If $|G_P^{(p+2)}|=1$, \eqref{rhg} provides
$$
2g-2=-2p^2+\sum_{i=0}^{p+1}\left(|G_P^{(i)}|-1 \right)=-2p^2+2(p^2-1)+p(p-1),
$$
which gives 
$$
p=\frac{1+\sqrt{8g+1}}{2}.
$$
If $|G_P^{(p+2)}|\neq 1$, from Proposition \ref{congp} we have 
$$
|G_P^{(2)}|=|G_P^{(3)}|=\cdots =|G_P^{(2p+1)}|=p, 
$$
and \eqref{rhg} implies 
$$
p \leq \frac{1+\sqrt{4g+1}}{2}<\sqrt{g}+1.
$$
In turn, suppose that $|G_P^{(2)}|=p^2$. If $|G_P^{(3)}|=1$, we obtain
$$
2g-2=-2p^2+\sum_{i=0}^{2}\left(|G_P^{(i)}|-1 \right)=p^2-3,
$$
whence $p=\sqrt{2g+1}$. 
If $|G_P^{(3)}|=p$, we must have $|G_P^{(3)}|=\cdots =|G_P^{(p+2)}|=p$. Then
$$
2g-2 \geq -2p^2+3(p^2-1)+p(p-1),
$$
i.e., $2p^2-p-(2g+1)\leq 0$. Hence $p \leq \frac{1+\sqrt{16g+9}}{4}<\sqrt{g}+1$. 
If $|G_P^{(3)}|=p^2$, the arguing as above we obtain $p \leq \sqrt{g+1}<\sqrt{g}+1$.
\item Assume $r=0$ and $s\geq 3$. Then from \eqref{rhg}, again taking into account that $|G_P^{(0)}|=|G_P^{(1)}|=p$ for all point $P\in \xx$ with nontrivial stabilizer in $G$, we obtain
$$
2g-2 \geq -2p^2+2sp(p-1)\geq -2p^2+6p(p-1),
$$
and so $2p^2 -3p-(g-1)\leq 0$. This last inequality provides 
$$
p \leq \frac{3+\sqrt{8g+1}}{4}<\sqrt{g}+1.
$$
\item Assume $r=0$ and $s=2$. As a direct application of \eqref{rhg} as in the previous cases, we conclude that $p^2-2p-(g-1) \leq 0$, which gives 
$$
p \leq \sqrt{g}+1.
$$
It is straightforward to check that equality $p=\sqrt{g}+1$ holds if, and only if, $|G_P^{(2)}|=1$ for all $P\in \xx$ with nontrivial stabilizer in $G$.
\item The case $r=0$ and $s=1$ does not occur. Indeed, in this case $G$ has a single short orbit $\Omega=\{P_1, \ldots,P_p\}$ on $\xx$. Since $G$ is abelian, we have that $G_{P_1}=G_{P_j}$ for all $j \in \{1,\ldots,p\}$. Set $H:=G_{P_1}$. Then $H$ has $p$ fixed points on $\xx$ and the cover $\xx/H \lra \xx/G$ is unramified. But this is a contradiction, since $\xx/G$ is rational, and a rational function field does not admit an unramified field extension by the Riemann-Hurwitz formula.
\end{itemize}
\end{itemize}

The proof of Theorem \ref{main1} is now complete. 

\section{Proof of Theorem \ref{main2}}\label{pm2}

In this section we give a proof of Theorem \ref{main2} by showing each item as a proposition. It is important to recall that $\xx$ is a curve of genus $g\geq 2$ defined over an algebraically closed field $\kk$ of characteristic $p \geq 0$, and the full automorphism group $\aut(\xx)$ of $\xx$ admits a subgroup $G$ of order $\ell^2$, where $\ell$ is a prime integer. Moreover, $\bar{g}$ is the genus of $\xx/G$ and $r$ and $s$ are respectively the number of fixed points and the number of orbits of size $\ell$ of $G$.

\begin{prop}\label{a1}
Let $\xx$ be a curve of genus $g \geq 2$ defined over an algebraically closed field $\kk$ of characteristic $p \geq 0$.  Assume that there exists a subgroup $G\leq \aut(\xx)$ such that $|G|=\ell^2$, where $\ell\neq p$ is a prime number. Then $\ell  = \frac{3+\sqrt{8g+1}}{2}$ if and only if $\xx$ is birationally equivalent to the Fermat curve $$\ff: x^\ell+y^\ell=1.$$
\end{prop}
\begin{proof}
First, the genus of $\ff$ is $g=(\ell-1)(\ell-2)/2$, as $\ff$ is smooth. Thus $\ell  = \frac{3+\sqrt{8g+1}}{2}$. Conversely, suppose that $\ell  = \frac{3+\sqrt{8g+1}}{2}$. By the proof of Theorem \ref{main1}, in subsection \ref{tc}, this equality occurs only if $\bar{g}=0$, $r=0$ and $s=3$. Hence $G$ has three short orbits of size $\ell$ on $\xx$, namely $\Omega_1$, $\Omega_2$ and $\Omega_3$. We first note that $G$ is not cyclic. In fact, if $G$ is cyclic, then there is a unique subgroup $H \leq G$ of order $\ell$. Thus every point of $\Omega_1\cup\Omega_2\cup\Omega_3$ is fixed by $H$, and if $g^\prime$ is the genus of $\xx/H$, \eqref{rhtame} gives 
$$
\ell^2-3\ell=\ell(2g^\prime-2)+3\ell(\ell-1),
$$
which implies $g^\prime=-\ell+1$, an absurd. Therefore, $G \cong \zz_\ell \times \zz_\ell$. Recall that since $G$ is abelian, for each $i=1,2,3$, every point of $\Omega_i$ has the same stabilizer in $G$. Let $P\in \Omega_1$; then $|G_P|=\ell$ and $G_P$ fixes every point of $\Omega_1$. The same argument via Riemann-Hurwitz formula used above to rule out the possibility of $G$ being cyclic shows that $G_P$ does not fix other points, that is, $G_P$ acts transitively on both $\Omega_2$ and $\Omega_3$. If $g_1$ denotes the genus of $\xx/G_P$, by \eqref{rhtame} applied to the cover $\xx\lra \xx/G_P$ we have
$$
\ell^2-3\ell=\ell(2g_1-2)+\ell(\ell-1) \Lra g_1=0.
$$
Arguing exactly in the same way, if $Q \in \Omega_2$ we have that $G_Q$ fixes precisely $\ell$ points, and then the genus of $\xx/G_Q$ is $g_2=0$. Since $G_P \neq G_Q$, we conclude that $G=G_P \times G_Q$ and both $\xx/G_P$ and $\xx/G_Q$ are rational curves; in other words, $\xx$ is a generalized Fermat curve over $\kk$.

 Let $x,y \in \kk(\xx)$ such that $\kk(\xx/G_P)=\kk(x)$ and $\kk(\xx/G_Q)=\kk(y)$. Since $G_P \cap G_Q=\{1\}$, it follows from basic Galois Theory that $\kk(x,y)=\kk(\xx)^{G_P \cap G_Q=}=\kk(\xx)$. Moreover, since both $\kk(\xx/G_P)$ and $\kk(\xx/G_Q)$ are Kummer extensions of degree $\ell$ of $\kk(\xx/G)$, $x$ and $y$ can be chosen such that $\kk(\xx/G)=\kk(x^\ell)=\kk(y^\ell)$. The extension $\kk(y):\kk(y^\ell)$ (resp. $\kk(x):\kk(x^\ell)$ ) has only two ramified points: the zero and the pole of $y^\ell$ (resp. $x^\ell$). Since each short orbit of $G$ lie over only one of this points, we may assume (without loss of generality) that $x^\ell$ and $y^\ell$ have a common pole and distinct zeros. Hence $y^\ell=a x^\ell+b$, with $a,b \in \kk$. A formal replacement of $b^{-1/\ell}y$ and $(-ab^{-1})^{1/\ell}x$ with $y$ and $x$ respectively finishes the proof. 
\end{proof}

\begin{rem}
The argument used in the last paragraph of the proof of Proposition \ref{a1} is adapted to our setting from that used in the proof of \cite[Theorem 5.1(a)]{AS2}. We included the details here for sake of completeness.
\end{rem}

We now deal with curves with $\ell\neq p$ attaining the second biggest possible value. 

\begin{prop}\label{a2}
Let $\xx$ be a curve of genus $g \geq 2$ defined over an algebraically closed field $\kk$ of characteristic $p \geq 0$.  Assume that there exists a subgroup $G\leq \aut(\xx)$ such that $|G|=\ell^2$, where $\ell\neq p$ is a prime number. Then $\ell  = \frac{1+\sqrt{8g+1}}{2}$ if and only if $\xx$ is birationally equivalent to a curve of the family  
$$
\mathcal{K}^{1}_{\alpha,\beta}: y^{\ell^2}=x^\alpha(x-1)^{\ell \beta},
$$
where $\alpha$ and $\beta$ are positive integers such that $\alpha +\beta \ell <\ell^2$ with $\ell \nmid \alpha$.
\end{prop}
\begin{proof}
By Kummer theory, see \cite[Proposition 3.7.3]{St}, the genus of $\mathcal{K}^{1}_{\alpha,\beta}$ is $g=\frac{\ell(\ell-1)}{2}$, whence $\ell  = \frac{1+\sqrt{8g+1}}{2}$. Conversely, assume that $\ell  = \frac{1+\sqrt{8g+1}}{2}$. From the proof of Theorem \ref{main1}, this equality holds only if $\bar{g}=0$, $r=2$ and $s=1$. Since $r>0$, we have that $G$ fixes a point; in particular, $G$ must be cyclic (see e.g.\cite[Proposition 3.8.5]{St}). Now $\ell^2>\ell^2-\ell+1=2g+1$. Further, note that $\ell^2 \not\equiv 2 \mod 4$. Thus by Theorem \ref{dgt}, $\xx$ is birationally equivalent to the curve
$$
\cc: y^{\ell^2}=x^\alpha(x-1)^\gamma
$$
where $\alpha$ and $\gamma$ are positive integers such that $\alpha +\gamma <\ell^2$ with $\gcd(\alpha,\gamma,\ell^2)=1$. The three short orbits of $G$ are over the points $P_0, P_1,P_\infty \in \p^1\cong \xx/G$, where $P_a=(a:1)$ and $P_\infty=(1:0)$, $a\in \kk$. Again using Proposition \cite[Proposition 3.7.3]{St}, we can infer that two of the values of the set  $\{\alpha, \gamma, \alpha+\gamma\}$ are co-prime with $\ell$ and one is divisible by $\ell$. If $\ell \mid \gamma$, then $\gamma=\ell \beta$ for some $\beta$, and the result follows. If $\ell \mid \alpha$ then the change of coordinates $(x,y)\mapsto (-x+1,y)$ gives the desired expression. If  $\ell \mid \alpha+\gamma$, then the result follows using the change of coordinates $(x,y)\mapsto (x^{-1},(-1)^{\beta/\ell^2}x^{-1}y)$. 
\end{proof}

We end the tame cases of Theorem \ref{main2} with the next result.

\begin{prop}\label{a3}
Let $\xx$ be a curve of genus $g \geq 2$ defined over an algebraically closed field $\kk$ of characteristic $p \geq 0$.  Assume that there exists a subgroup $G\leq \aut(\xx)$ such that $|G|=\ell^2$, where $\ell\neq p$ is a prime number. Then $\ell=\sqrt{2g+1}$ if, and only if, $\xx$ is either birationally equivalent to a curve of the family 
$$
\mathcal{K}^{2}_{\alpha,\beta}: y^{\ell^2}=x^\alpha(x-1)^{\beta},
$$
where $\alpha$ and $\beta$ are positive integers such that $\alpha +\beta <\ell^2$ with $\alpha$, $\beta$ and $\alpha+\beta$ co-prime with $\ell$, or birationally equivalent to a curve of the family
$$
\mathcal{G}_a:ax^3y^3+x^3+y^3=1,
$$
  $a \in \kk \backslash \{0\}$ or birationally equivalent to a curve of the family
$$
\mathcal{N}_a:(y^3-x^3)^2+(a-1)(x^3-ay^3)=0 
$$
with $a \in \kk \backslash \{0,1\}$.
\end{prop}
\begin{proof}
It is not difficult to verify that the three curves above have genus equal to $(\ell^2-1)/2$. We now deal with the converse; that is, assume that $\aut(\xx)$ has a subgroup $G$ of order  $\ell^2=2g+1$.
According to the proof of Theorem \ref{main1}, the equality $\ell=\sqrt{2g+1}$ is possible only in three situations: 
\begin{itemize}
\item[(1)] $\bar{g}=1$, $r=0$, $s=1$, $\ell=3$ and $g=4$.
\item[(2)] $\bar{g}=0$, $r=3$, $s=0$.
\item[(3)] $\bar{g}=0$, $r=0$, $s=4$, $\ell=3$ and $g=4$.
\end{itemize}
First we show that case (1) actually cannot occur. In such a case we have that $G$ is not cyclic. As we saw in the proof of Theorem \ref{main1}, there exists $H \leq G$ of order $\ell$ such that $\xx/H$ has genus $\tilde{g}=2$ and $\xx \lra \xx/H$ is unramified. Since the cover $\xx \lra \xx/G$ has only one branch point, we conclude that the cover $\xx/H \lra \xx/G$ has a unique branch point as well. However, the cover $\xx/H \lra \xx/G$ is cyclic and tame, and then it corresponds to a Kummer cover with only one branch point, which is impossible according to \cite[Proposition 10]{DGT}.

Now suppose that (2) holds, i.e., $\bar{g}=0$, $r=3$, $s=0$. Since $r>1$, we have that $G$ is cyclic. Further, $|G|=\ell^2=2g+1$. Thus by Theorem \ref{dgt}, $\xx$ is birationally equivalent to the curve
$$
\cc: y^{\ell^2}=x^\alpha(x-1)^\beta
$$
where $\alpha$ and $\beta$ are positive integers such that $\alpha +\beta <\ell^2$ with $\gcd(\alpha,\beta,\ell^2)=1$. Here the points $(0:1),(1:1),(1:0) \in \p^1\cong \xx/G$, must be fully ramified. By Proposition \cite[Proposition 3.7.3]{St} we conclude that the values of the set  $\{\alpha, \gamma, \alpha+\gamma\}$ are all co-prime with $\ell$. 

Finally assume (3), that is, $\bar{g}=0$, $r=0$, $s=4$, $\ell=3$ and $g=4$. First, $G$ is not cyclic. In fact, if $G$ is cyclic then there is only one subgroup $H \leq G$ of order $3$, and such subgroup must fix every point on a short orbit of $G$. Using this information in the Riemann-Hurwitz formula \eqref{rhtame} leads us to a contradiction. Denote by $\Omega_i$ the short orbits of $G$ on $\xx$ for $i=1,2,3,4$. The group $G$ is abelian, so for a fixed $i$, every point on $\Omega_i$ share the same stabilizer in $G$. Again from \eqref{rhtame}, the points on at most two such short orbits of $G$ can share the same stabilizer in $G$. We then have two possible scenarios, without loss of generality: 
\begin{itemize}
\item[(I)] There are $H,K \leq G$ with $H\neq K$ such that the fixed points of $H$ are precisely  $\Omega_1 \cup \Omega_2$ and the fixed points of $K$ are precisely  $\Omega_3 \cup \Omega_4$. 
\item[(II)] There are $H,K \leq G$ with $H\neq K$ such that the fixed points of $H$ are precisely  $\Omega_1$ and the fixed points of $K$ are precisely  $\Omega_2$.
\end{itemize}
In case (I), the Riemann-Hurwitz formula provides that both $\xx/H$ and $\xx/K$ are rational, that is, $\xx$ is a generalized Fermat curve. An argument analogous to the one given in the proof of Proposition \ref{a1} (see also \cite[Theorem 5.1]{AS2}) implies that $\xx$ is birationally equivalent to $ax^3y^3+x^3+y^3=1$, with $a \in \kk^*$.

In case (II) the Riemann-Hurwitz formula provides that both $\xx/H$ and $\xx/K$ are elliptic curves, that is, both have genus $1$. Let $u \in \kk(\xx)$ such that $\xx/G=\kk(u)$. Both covers $\xx/H\lra \xx/G$ and $\xx/K\lra \xx/G$ are fully ramified at three points. Since $\xx \lra \xx/G$ has four branch points, we have that the covers $\xx/H\lra \xx/G$ and $\xx/K\lra \xx/G$  must share two branch points. Since $\aut(\p^1)\cong \operatorname{PGL}(2,\kk)$  is sharply $3$-transitive on the points of $\p^1$, we may assume that the cover $\xx/H\lra \xx/G$ ramifies at $P_0,P_1$ and $P_\infty\in \p^1$, and $\xx/K\lra \xx/G$  ramifies at $P_0,P_a$ and $P_\infty$ for some $a\in \kk\backslash \{0,1\}$. The extension $\kk(\xx/H):\kk(\xx/G)$ is a Kummer extension of degree $3$, and since it ramifies precisely at $P_0,P_1,P_\infty\in \p^1$, one can see from \cite[Proposition 3.7.3]{St} that there exists $y \in \kk(\xx/H)$ such that $\kk(\xx/H)=\kk(u,y)$ and
$
y^3=u^r(u-1)^s
$
for some nonzero integers $r,s$ such that $r,s$ and $r+s$ are co-prime with $3$. Up to a replacement of $y$ with $zy$ for a suitable $z \in \kk(u)$, we may assume that $r$ and $s$ are positive and $r+s<3$, which means that $r=s=1$. In other words, $\kk(\xx/H)=\kk(u,y)$ with
$$
y^3=u(u-1).
$$
 Using an analogous argument and taking into account that $\kk(\xx/K):\kk(\xx/G)$ is a Kummer extension of degree $3$ ramifying precisely at $P_0,P_a,P_\infty\in \p^1$, $a\in \kk\backslash \{0,1\}$, we conclude that there exists $x \in \kk(\xx)$ such that $\kk(\xx/K)=\kk(u,x)$ with
$$
x^3=u(u-a).
$$ 
Note that since $x \notin \kk(\xx/H)$ (otherwise we would have $\kk(\xx/K) =\kk(\xx/H)$) and $[\kk(\xx):\kk(\xx/H)]=3$, we conclude that $\kk(\xx)=\kk(x,y)$. Furthermore, 
$$
u^2=y^3+u=x^3+au \Lra u=\frac{y^3-x^3}{a-1}.
$$
Hence, replacing the last equality in $y^3=u^2-u$, we obtain
\begin{equation}\label{extra}
(y^3-x^3)^2+(a-1)(x^3-ay^3)=0.
\end{equation}
It remains to show that \eqref{extra} defines an irreducible curve. Since $\kk(\xx)=\kk(x,y)$, a plane model for $\xx$ is defined by an irreducible factor of the left side of \eqref{extra}. Since $\xx$ has genus $g=4$, such a factor must have degree at least $d \geq 5$, as $g\leq (d-1)(d-2)/2$. However, if $(y^3-x^3)^2+(a-1)(x^3-ay^3)$ has a degree $5$ factor, then it has a linear factor, and a straightforward computation shows that this does not occur. Therefore $d=6$, and this finishes the proof.
\end{proof}

In the sequel of this section we turn our attention to the wild case, that is, $\ell=p=\Char(\kk)$. Since we are assuming that $g\geq2$ and $2<\sqrt{2}+1\leq \sqrt{g}+1$, in the rest of this section we may assume $p>2$.

\begin{prop}\label{b1}
Let $\xx$ be a curve of genus $g \geq 2$ defined over an algebraically closed field $\kk$ of characteristic $p \geq 0$.  Assume that there exists a subgroup $G\leq \aut(\xx)$ such that $|G|=p^2$. Then  $p=\frac{1+\sqrt{8g+1}}{2}$ if and only if $\xx$ is birationally equivalent to a curve of the family
$$
\yy_{a}:y^p-y=x^{p+1}+ax^2, 
$$
with $a \in \kk$. 
\end{prop}
\begin{proof}
By \cite[Proposition 3.7.8]{St} we immediately have that $\yy_{a}$ has genus $g=p(p-1)/2$, whence $p=\frac{1+\sqrt{8g+1}}{2}$. Moreover, it will follow from Theorem \ref{main3} that $\aut(\yy_{a})$ admits a subgroup of order $p^2$. 

Conversely, assume that $p=\frac{1+\sqrt{8g+1}}{2}$. Then $p^2>2g+1$. In particular, $p>2$ and by \cite[Theorem 15]{DGT}, $G$ is not cyclic. By the proof of Theorem \ref{main1}, $p$ assumes this value only if $\bar{g}=0$, $r=1$ and $s=0$, where $\bar{g}$, $r$ and $s$ are respectively the genus of $\xx/G$, the number of short orbits of $G$ of size $1$ and the number of short orbits of $G$ of size $p$. 
Since $G \cong \zz_p \times \zz_p$, there are precisely $p+1$ distinct subgroups $H_1,\ldots,H_{p+1}$ of $G$ of order $p$. For each $i \in\{1,\ldots,p+1\}$, set $\Ei:=\kk(\xx)^{H_i}$ and denote by $g(\Ei)$ the genus of $\Ei$. Further, let  $u \in \kk(\xx)$ such that $\kk(\xx)^G=\kk(u)$. Then for each $i$, $\Ei/\kk(u)$ is a cyclic extension with $[\Ei:\kk(u)]=p$, that is, $\Ei/\kk(u)$ is an Artin-Schreier extension. We may assume that the unique ramified point $P\in \xx$ lies over $P_\infty=(1:0) \in \xx/G$. From \cite[Proposition 3.7.8]{St} for all $i\in \{1,\ldots,p+1\}$ there exists a positive integer $m_i$ with $p\nmid m_i$ such that
$$
g(\Ei)=\frac{p-1}{2}(m_i-1).
$$
As a matter of fact, such $m_i$ is defined as follows: if $P^{(i)}_\infty \in \xx/H_i$ is the unique point over $P_\infty$, then the different exponent of $P^{(i)}_\infty$ is given by $d_{P^{(i)}_\infty}=(p-1)(m_i+1)$. By \cite[Theorem 2.1]{GS} we have
\begin{equation}\label{gen}
g=\sum_{i=1}^{p+1}g(\Ei).
\end{equation}
Therefore
\begin{equation}\label{gen2}
\frac{p(p-1)}{2}=g=\sum_{i=1}^{p+1}g(\Ei)=\sum_{i=1}^{p+1}\frac{p-1}{2}(m_i-1) \ \ \Lra \ \  p=\sum_{i=1}^{p+1}(m_i-1).
\end{equation}
From \eqref{gen2} we have that $m_i=1$ for at least one $i$. Without loss of generality, let us suppose that $m_1=1$. Then $g(\E_1)=0$, that is, $\E_1$ is a rational function field. The Artin-Schreier Theory (see e.g. \cite[Section 2]{Ma}) yields that there exists $z \in \E_1$ such that $\E_1=\kk(u,z)$ with
\begin{equation}\label{as}
z^p-z=f(u)=\frac{h(u)}{\prod\limits_{i=1}^{k}(u-a_i)^{\lambda_i}},
\end{equation}
where $p \nmid \lambda_i$ for $i\in\{1,\ldots,k\}$ and $h(u) \in \kk[u]$ is such that $\deg(h(u))-\sum _{i=1}^{k}\lambda_i$ is either negative, zero or relatively prime to $p$ (in this case $f(u)$ is said to be in the standard form). It should be noted that such expression for $f(u)$ can be obtained only adjusting the variable $z$. By \cite[Proposition 3.7.8]{St}, $-\lambda_i=v_{P_{a_i}}(f(u))=0$ and $-1=-m_1=v_{P_\infty}(f(u))=-\deg(h(u))$, where $v_Q$ denotes the discrete valuation at the place $Q$. Hence $f(u)=\alpha u+\beta$ for $\alpha, \beta \in \kk$, $\alpha \neq 0$. If $\gamma \in \kk$ is a root of the polynomial $T^p-T-\beta$, up the formal replacement of $z-\gamma$ with $z$ and $\alpha u$ with $u$, we obtain that $\E_1=\kk(z)$ with
$$
z^p-z=u.
$$
We now claim that $m_i>1$ for all $i\in\{2,\ldots,p+1\}$. In fact, if $m_i=1$ for some such $i$, say $i=2$, we have that $g(\E_2)=0$ and arguing exactly as above we conclude that $\E_2=\kk(w)$, where $w^p-w=\xi u+\delta$ for $\xi, \delta \in \kk$, $\xi \neq 0$, for some $w \in \E_2$. Let $\eta \in \kk$ be a root of $T^p-T-\delta$ and set $\tilde{w}:=w-\xi^{1/p}z-\eta$. Since $\kk(\xx)=\E_1 \cdot \E_2$, we have that $\kk(\xx)=\kk(z,w)=\kk(z,\tilde{w})$. Now 
\begin{eqnarray}
\tilde{w}^p-\tilde{w} & = & (w-\xi^{1/p}z-\eta)^p-(w-\xi^{1/p}z-\eta) \nonumber \\
                               &=& (w^p-w)-(\xi z^p-\xi^{1/p}z)-(\eta^p-\eta) \nonumber\\
                               &=& \xi (z^p-z)+\delta-\xi z^p + \xi^{1/p}z -\delta \nonumber\\
                               &=& (\xi^{1/p}-\xi)z, \nonumber
\end{eqnarray}
implying that $\kk(\xx)=\kk(z,\tilde{w})=\kk(\tilde{w})$, which is impossible since $\xx$ has genus $g\geq 2$. Hence the claim is proved. From $m_1=0$, $m_i>1$ for all $i\in\{2,\ldots,p+1\}$ and \eqref{gen2} we conclude that $m_i=2$ for all $i\in\{2,\ldots,p+1\}$. In particular, $m_2=2$. Let $w \in \E_2$ such that $\E_2=\kk(u,w)$ and $w^p-w=f(u)$ for some $f(u) \in \kk(u)$ of the same type of the right side of \eqref{as}. Since $m_2=2$, arguing as we did to obtain a generating equation for $\E_1$, we conclude that $\E_2=\kk(u,w)$ with
\begin{equation}\label{eqqf}
w^p-w=\alpha_2u^2+\alpha_1u+\alpha_0, 
\end{equation}
for $\alpha_i \in \kk$, $\alpha_2 \neq 0$. Note that $\kk(\xx)=\E_1 \cdot \E_2=\kk(z)\cdot \kk(u,w)=\kk(z,w)$, as $u=z^p-z \in \kk(z)$. Replacing $u=z^p-z $ in \eqref{eqqf} gives 
\begin{eqnarray}
w^p-w&=&\alpha_2(z^p-z)^2+\alpha_1(z^p-z)+\alpha_0 \nonumber \\
         &=& (\alpha_2^{1/p}z^2+\alpha_1^{1/p}z)^p-2\alpha_2z^{p+1}+\alpha_2z^2-\alpha_1z+\alpha_0.\nonumber
\end{eqnarray}
Define $\bar{w}:=w-(\alpha_2^{1/p}z^2+\alpha_1^{1/p}z)$. Then $\kk(\xx)=\kk(z,\bar{w})$ and
$$
\bar{w}^p-\bar{w}=-2\alpha_2z^{p+1}+(\alpha_2^{1/p}+\alpha_2)z^2+(\alpha_1^{1/p}-\alpha_1)z+\alpha_0.
$$
Now let $x:=(-2\alpha_2)^{1/(p+1)}z$ and $y:=\bar{w}-\zeta$, where $\zeta \in \kk$ is a root of $T^p-T-\alpha_0$. Then $\kk(\xx)=\kk(x,y)$ where
$$
y^p-y=x^{p+1}+ax^2+bx,
$$
for certain $a,b \in \kk$. Finally, let $e \in \kk$ be a root of the polynomial $T^{p^2}+2aT^p+T+b$ and $f \in \kk$ a root of $T^p-T+(e^{p^2+p}+ae^{2p})$. Then defining $\tilde{x}=x-e^p$ and $\tilde{y}=y-ex-f$, we obtain that $\kk(x,y)=\kk(\tilde{x},\tilde{y})$ where
$$
\tilde{y}^p-\tilde{y}=\tilde{x}^{p+1}+a\tilde{x}^2.
$$
This finishes the proof.
\end{proof}

To deal with the characterization of the curves attaining the second largest bound in wild case, we need the following result.

\begin{lem}\label{bnz}
Assume $p>2$.The curve $\cc_{a,b}: \big((x^p-x)-a(y^p-y)\big)^2=b^2(y^p-y)$
where $a,b \in \kk$, with $a \notin \mathbb{F}_p$ and $b \neq 0$ has genus $g=(p^2-1)/2$.
\end{lem}
\begin{proof}
This follows by an argument analogous to the one used in \cite[Example 3.9]{BNZ}. For sake of completeness,  we will include the proof here. Consider the conic $\cc: f(u,v)=0$, where $f(u,v)=(v-au)^2-b^2u$. Then $\cc$ can be parametrized by $u=t^2$ and $v=at^2+bt$, whence $\kk(\cc)=\kk(u,v)=\kk(t)$. We clearly have that $\cc_{a,b}$ is defined by the double Artin-Schreier cover of $\cc$ given by
 \[ \begin{cases}
y^p-y=u\\
x^p-x=v\\
f(u,v)=0.
\end{cases}\]
One can show that $\cc_{a,b}$ is irreducible via \cite[Lemma 2.2]{BNZ}. If $\{\mu_1,\ldots,\mu_{p+1}\} \subset \F_{p^2}^*$ denotes a set of representatives for the cosets of $\F_{p^2}^*$ modulo $\F_{p}^*$ and $\eta \in \F_{p^2}\backslash \F_p$, then by \cite[Section 2]{BNZ} the intermediate fields $\kk(t) \subset \Ei \subset \kk(x,y)$ such that $[\Ei:\kk(t)]=p$ are defined by $\Ei=\kk(t,w_i)$, where
$$
w_ i^p-w_i=(\mu_i^p-\mu_i+a((\mu_i \eta)^p-\mu_i \eta))t^2+((\mu_i \eta)^p-\mu_i \eta)bt.
$$
As $a \notin \F_p$, one can check that that $\mu_i^p-\mu_i+a((\mu_i \eta)^p-\mu_i \eta)\neq 0$ for all $i$. Thus by \cite[Proposition 3.7.8]{St} we have $g(\Ei)=(p-1)/2$ for $i=1,\ldots,p+1$. Therefore, from \cite[Theorem 2.1]{GS} we obtain
$$
g=\sum_{i=1}^{p+1}g(\Ei)=\frac{p^2-1}{2}.
$$
\end{proof}

\begin{prop}\label{b2}
Let $\xx$ be a curve of genus $g \geq 2$ defined over an algebraically closed field $\kk$ of characteristic $p \geq 0$.  Assume that there exists a subgroup $G\leq \aut(\xx)$ such that $|G|=p^2$. Then  $p=\sqrt{2g+1}$ if and only if $p>2$ and $\xx$ is birationally equivalent to a curve of the family
$$
\cc_{a,b}: \big((x^p-x)-a(y^p-y)\big)^2=b^2(y^p-y),
$$
where $a,b \in \kk$, with $a \notin \mathbb{F}_p$ and $b \neq 0$.
\end{prop}
\begin{proof}
The ``if'' part follows from Lemma \ref{bnz}. Now assume that $p=\sqrt{2g+1}$. It follows again from \cite[Theorem 15]{DGT} that $G$ is not cyclic. By the proof of Theorem \ref{main1}, $p$ assumes this value only if $\bar{g}=0$, $r=1$ and $s=0$, where $\bar{g}$, $r$ and $s$ are respectively the genus of $\xx/G$, the number of short orbits of $G$ of size $1$ and the number of short orbits of $G$ of size $p$. Moreover, if $P\in \xx$ is the unique ramified point, then 
\begin{equation}\label{ram}
|G_P^{(0)}|=|G_P^{(1)}|=|G_P^{(2)}|=p^2  \ \ \text{ and } \ \ |G_P^{(3)}|=1.
\end{equation}
Set $\kk(\xx)^G=\kk(t)$ and, as in the proof of Proposition \ref{b1}, for each $i=1,\ldots,p+1$ let $\Ei$ denote the intermediate fields $\kk(t) \subset \Ei \subset \kk(\xx)$ such that $[\Ei:\kk(t)]=p$. Then we have
$$
g(\Ei)=\frac{p-1}{2}(m_i-1).
$$
We claim that $m_i = 2$ for all $i=1,\ldots,p+1$. Indeed, if $P^{(i)}_\infty \in \xx/H_i$ is the unique point over $P_\infty$, then $d_{P^{(i)}_\infty}=(p-1)(m_i+1)$ (here $H_i \leq G$ is such that $\kk(\xx)^{H_i}=\Ei$). By the transitivity of the different exponent  (see \cite[Corollary 3.4.12]{St}), we have
\begin{equation}\label{diff}
d_P=|H_i| \cdot d_{P^{(i)}_\infty}+\sum_{j=0}^\infty \left( |(H_i)_P^{(j)}|-1\right).
\end{equation}
By \eqref{ram} we have $d_P=3(p^2-1)$. Since  $(H_i)_P^{(j)}=H_i \cap G_P^{(j)}$, then \eqref{ram} implies that $|(H_i)_P^{(0)}|=|(H_i)_P^{(1)}|=|(H_i)_P^{(2)}|=p$ and $|(H_i)_P^{(3)}|=1$. Now replacing $d_{P^{(i)}_\infty}=(p-1)(m_i+1)$ in \eqref{diff} gives the claim.

As in Proposition \ref{b1}, we have that $\kk(\xx)=\E_1 \cdot \E_2$, and there exist $x,y \in \kk(\xx)$, $a,b,c,d \in \kk$ with $a,c \neq 0$ such that $\E_1=\kk(t,y)$, $\E_2=\kk(t,x)$ where $y^p-y=ct^2+dt$ and 
$$x^p-x=at^2+bt.
$$
 Up to a scaling of $t$ followed by completing the squares and a translation of $y$, we may assume that $\E_1=\kk(t,y)$ with
\begin{equation}\label{t2}
y^p-y=t^2.
\end{equation}
Note that such a change of coordinates forces $b \neq 0$. Indeed, if $b=0$ then $x^p-x=a(y^p-y)$, which implies that $\kk(x,y)$ is a rational function field. However, in this case $\kk(t)\subsetneq \kk(x,y)\subsetneq \kk(\xx)$, thus $\kk(x,y)=\E_j$ for some $j$, contradicting the fact that $g(\E_j)=(p-1)/2$. In particular, 
$$
t=b^{-1}((x^p-x)-a(y^p-y)) \in \kk(x,y),
$$
whence $\kk(\xx)=\kk(t,x,y)=\kk(x,y)$. Moreover, replacing the equality above in \eqref{t2}, we obtain
$$
 \big((x^p-x)-a(y^p-y)\big)^2=b^2(y^p-y).
$$
It remains to prove that $a \notin \F_p$. If so, then $a^p=a$ and setting $z=x-ay$ we obtain $\kk(x,y)=\kk(y,z)$ where
\begin{equation}\label{fami}
(z^p-z)^2=b^2(y^p-y).
\end{equation}
It is not difficult to show that the curve defined by the equation \eqref{fami} is isomorphic to a curve of the family given in Proposition \ref{b1}, which is a contradiction since $\xx$ has genus $(p^2-1)/2$. Therefore $a \notin \F_p$.
\end{proof}

We end this section with the characterization of curves of genus $g\geq 2$ admitting automorphism groups of order $p^2=(\sqrt{g}+1)^2$.

\begin{prop}\label{b3}
Let $\xx$ be a curve of genus $g \geq 2$ defined over an algebraically closed field $\kk$ of characteristic $p \geq 0$.  Assume that there exists a subgroup $G\leq \aut(\xx)$ such that $|G|=p^2$. Then  $p=\sqrt{g}+1$ if and only if $p>2$ and $\xx$ is birationally equivalent to a curve of the Artin-Mumford type
$$
\mathcal{M}_c: (y^p-y)(x^p-x)=c,
$$
with $c\in \kk^*$.
\end{prop}
\begin{proof}
It is well known that the Artin-Mumford curve has a subgroup $G$ of order $p^2$ and genus $g=(p-1)^2$, see e.g. \cite[Proposition 7]{AK} (as a matter of fact, this is not difficult to check). So let us assume that $\xx$ is a curve of genus $g =(p-1)^2 \geq 2$ having a subgroup of automorphisms $G$ of order $p^2$. From the proof of Theorem \ref{main1}, $p=\sqrt{g}+1$ holds only if $\bar{g}=0$, $r=0$, $s=2$ and $G_P^{(2)}$ is trivial for all $P \in \xx$. Let us denote the two short orbits of size $p$ of $G$ on $\xx$ by $\Omega_1$ and $\Omega_2$. We first show that $G$ is not cyclic. To this end, assume on the contrary, that $G$ is cyclic.  Then the unique subgroup $H \leq G$ of order $p$ must fix every point on $\Omega_1 \cup \Omega_2$. Consequently, the cover $\xx/H \lra \xx/G$ is unramified, which is a contradiction since $\xx/G$ is a rational curve. Thus $G\cong \zz_p \times \zz_p$. Since $G$ is an abelian group, every point on $\Omega_1$ share the same stabilizer in $G$, say $H_1$. If $H_2$ is the common stabilizer in $G$ of the points on $\Omega_2$, then $H_1 \neq H_2$, otherwise we would have again that $\xx/H_1 \lra \xx/G$ is unramified. For $i=1,2$, since $H_i \leq G_P$  for $P \in \Omega_i$, we have that $(H_i)_P^{(j)}=H_i \cap G_P^{(j)}$ for all such $P$ and for all $j \geq 0$. In particular, we conclude that
$$
|(H_i)_P^{(0)}|=|(H_i)_P^{(1)}|=p \ \ \text{ and } \ \  |(H_i)_P^{(2)}|=1
$$
for $i=1,2$. Hence the Riemann-Hurwitz formula \eqref{rhg} implies that both $\xx/H_1$ and $\xx/H_2$ are rational curves. It follows directly from \cite[Theorem 6]{VM} that $\kk(\xx)=\kk(x,y)$, where either 
$$
(y^p-y)(x^p-x)=c, \text{ with $c\in \kk^*$ },
$$
or
$$
y^3-y=\frac{i}{x(x-1)}, \text{ where $i^2=2$}, \ \ \text{ or } y^p-y=\frac{1}{x^\lambda}, \text{ with $\lambda \mid p+1$}.
$$
By computing the genera of the last two curves above via \cite[Proposition 3.7.8]{St}, we obtain respectively $g=(p-1)(\lambda-1)/2$ and $g=2$, so these possibilities can be ruled out. This finishes the proof.
\end{proof}

\begin{rem}
By a direct application of the Deuring-Shafarevich formula \eqref{dsg} to the cover $\xx \lra \xx/G$ we obtain that the curves characterized in Propositions \ref{b1} and \ref{b2} have zero $p$-rank. Therefore, if one requires that a curve $\xx$ of genus $g\geq 2$ defined over an algebraically closed field $\kk$ of characteristic $p >2$ has positive $p$-rank and subgroup $G\leq \aut(\xx)$ such that $|G|=p^2$, then  $p\leq\sqrt{g}+1$ and equality holds if and only if $\xx$ is birationally equivalent to the Artin-Mumford curve $\mathcal{M}_c: (y^p-y)(x^p-x)=c,$ with $c\in \kk^*$.
\end{rem}

 \section{Proof of Theorem \ref{main3}}\label{pm3}
 
 Let $\ff$ be a plane curve and for $P \in \ff$, denote by $\mathcal{T}_P$ the tangent line to $\ff$ at $P$ . We say that the curve $\ff$ is classical if for a general point $P \in \ff$ the intersection multiplicity of $\ff$ with $\mathcal{T}_P$ at $P$ is $2$. Otherwise, we say that $\ff$ is nonclassical. Let $\kk(\ff)=\kk(x,y)$. Then at least one of the elements $x$ and $y$ is separating. If $x$ is a separating element, it follows from \cite{SV} that $\ff$ is nonclassical if, and only if, $D_x^{(2)}y=0$, where $D_x^{(j)}f$ stands for the $j$-th Hasse derivative of the function $f \in \kk(\ff)$ with respect to $x$.
 
 \begin{lem}\label{nc}
 Assume that $\Char(\kk)=p>2$ and let $\yy_{a}:y^p-y=x^{p+1}+ax^2$. Then $\yy_a$ is nonclassical if, and only if, $a=0$.
 \end{lem}
 \begin{proof}
  Applying $D_x^{(1)}$ in both sides of equality $y^p-y=x^{p+1}+ax^2$ gives
 \begin{equation}\label{hasse}
 - D_x^{(1)} y=x^p+2ax.
 \end{equation}
 Applying $D_x^{(1)}$ in both sides of \eqref{hasse} implies
 $
 - D_x^{(1)}(D_x^{(1)} y)=2a.
 $
 Taking into account that $D_x^{(1)} \circ D_x^{(1)}=2D_x^{(2)}$, we obtain $D_x^{(2)}y=-a$. Result then follows from the statement in the first paragraph of the section. 
 \end{proof}
 
 Using the lemma above, we can obtain the following.
 
 \begin{prop}\label{isonc}
  Assume that $\Char(\kk)=p>2$ and let $\yy_{a}:y^p-y=x^{p+1}+ax^2$. Then $\yy_a$ birationally equivalent to the Hermitian curve $\mathcal{H}:v^p+v=u^{p+1}$ if, and only if, $a=0$.
 \end{prop}
 \begin{proof}
 Suppose  $a=0$. Let $c \in \kk$ such that $c^{p-1}=-1$ and let $d \in \kk$ such that $d^{p+1}=-c$. Then the map $\varphi:\mathcal{H}\lra \yy_0$  given by $(u,v) \mapsto (du,cv)$ is birational. Conversely, assume that $\mathcal{H}$ is birationally equivalent to $\yy_a$. Then by Theorem \ref{chang}  $\mathcal{H}$ and $\yy_a$ are projectively equivalent. Since intersection multiplicities are preserved by projectivity and $\mathcal{H}$ is nonclassical, it follows that $\yy_a$ must be nonclassical. By Lemma \ref{nc} we conclude that $a=0$.
 \end{proof}
 
 Finally, we provide the full automorphism group of $\yy_a$ for $a\neq 0$.
 
 \begin{thm}
 Assume that $\Char(\kk)=p>2$ and let $\yy_{a}:y^p-y=x^{p+1}+ax^2$ with $a \neq 0$. Then $\aut(\yy_{a})$ fixes the only point at the infinity of $\yy_{a}$. More precisely, $\aut(\yy_{a})$ contains a unique normal Sylow $p$-subgroup $S_p$ which is isomorphic to the Heisenberg group of order $|S_p|=p^3$, and 
$$
\frac{\aut(\yy_{a})}{S_p} \cong \zz_{p-1}.
$$
 \end{thm} 
 \begin{proof}
 Since $g=(p^2-p)/2$ and $p^2=|G|=|G_P^{(0)}|=|G_P^{(1)}|$, where $P=(0:1:0) \in \yy_a$  is the unique ramified point of $\yy_a \lra \yy_a/G$, we conclude that $|G_P^{(1)}|>2g+1$. Further, $G_P^{(1)} \leq (\aut(\yy_a))_P^{(1)}$ implies that $|(\aut(\yy_a))_P^{(1)}|>2g+1$. Then \cite[Theorem 11.140]{HKT} implies that either $\aut(\yy_a)=(\aut(\yy_a))_P$ or $\yy_a$ is birationally equivalent to one of the following curves:
 \begin{itemize}
 \item[(i)] The Hermitian curve $y^q+y=x^{q+1}$, where $q=p^h$, $h>0$ and $g=q(q-1)/2$.
 \item[(ii)] The DLS curve $y^q+y=x^{q_0}(x^q+x)$ where $p=2$, $q_0=2^r$ and $q=2q_0^2$, $g=q_0(q-1)$.
 \item[(iii)] The DLR curve $y^{n^2}-(1+(x^n-x)^{n-1})y^n+(x^n-x)^{n-1}y-x^n(x^n-x)^{n+3n_0}=0$ with $p=3$, $g=n_0(n-1)$, where $n_0=3^r$, $n=3n_0^2$, $r\geq 0$.
 \end{itemize}
 Since $a\neq 0$, (i) can be ruled out by Proposition \ref{isonc}, and (ii) is ruled out as well as we are assuming $p>2$. For $p=3$,  the genus of $\yy_a$ is $g=3$; thus (iii) can be ruled out as the DLR curve has even genus. Therefore, we conclude that $\aut(\yy_a)=(\aut(\yy_a))_P$. In particular, the Sylow $p$-subgroup $S_p$ of $\aut(\yy_a)$ is normal and $\aut(\yy_a)/S_p$ is cyclic of order prime to $p$. Now, note that $\yy_a$ is nonsingular. Then $\aut(\yy_a)$ can be seen as a subgroup of $\operatorname{PGL}(2,\kk)$. As $\aut(\yy_a)$ fixes $P=(0:1:0)$, we have that every automorphism of $\yy_a$ (in projective coordinates) is of the form
 \begin{equation}\label{autproj}
 (X:Y:Z) \mapsto (\alpha X+\beta Z: \kappa X+\delta Y+\epsilon Z: \phi X+\gamma Z),
 \end{equation}
 with $\alpha, \beta, \kappa, \delta, \epsilon,\phi, \gamma \in \kk$. Using the fact that right side of \eqref{autproj} is submitted to $Y^pZ-YZ^p=X^{p+1}+aX^2Z^{p-1}$, a straightforward computation shows that we must have $\phi=0$. Hence we may assume that $\gamma=1$ and every element of $\aut(\yy_a)$ (in affine coordinates) is of the form
 \begin{equation}\label{autaff}
 (x,y)\mapsto (\alpha x+\beta , \kappa x +\delta y+\epsilon).
 \end{equation}
 Again, the right side of \eqref{autaff} is submitted to $y^p-y=x^{p+1}+ax^2$. Thus we conclude that $\alpha \in \F_p^{*}$, $\delta=\alpha^2$,  $\beta=\kappa^p/\alpha$ and every automorphism of $\yy_a$ is given by
 \begin{equation}
 \sigma_{\alpha,\kappa,\epsilon}(x,y)=(\alpha x +\kappa^p/\alpha, \kappa x+\alpha^2 y+\epsilon),
 \end{equation}
 where $\kappa \in \kk$ is a root of the polynomial $h(T)=T^{p^2}+2aT^p+T$ and 
 $$
 \epsilon^p-\epsilon=\frac{\kappa^{p(p+1)}+a\kappa^{2p}}{\alpha^2}.
 $$
One can check directly that $\sigma_{1,\kappa,\epsilon}$ has order $p$ for all $\kappa, \epsilon$. Since the polynomial $h(T)$ is separable, it has $p^2$ distinct roots, and for each root $\kappa$ of $h(T)$ there are $p$ distinct possible values for $\epsilon$. Hence $\{\sigma_{1,\kappa,\epsilon} \ | \ h(\kappa)=0  \text{ and } \epsilon^p-\epsilon=\kappa^{p(p+1)}+a\kappa^{2p}\}$ is a subgroup of $S_p$ of order $p^3$. On the other hand, the Deuring-Shafarevich formula \eqref{dsg} applied to the cover $\yy_a \lra \yy_a/G$ proves that $\yy_a$ has zero $p$-rank. Thus the Nakajima bound for $S_p$ (\cite[Theorem 1]{Na}) reads
$$
|S_p| \leq \max\left\{g, \frac{4p}{p-1}^2g^2\right\}=p^3.
$$
Therefore, $S_p=\{\sigma_{1,\kappa,\epsilon} \ | \ h(\kappa)=0  \text{ and } \epsilon^p-\epsilon=\kappa^{p(p+1)}+a\kappa^{2p}\}$.  Now $\sigma_{1,\kappa,\epsilon},\sigma_{1,c,e}\in S_p$ commute if and only if $\{c,\kappa\}$ is linearly dependent over $\F_p$. Since both $\kappa$ and $c$ are roots of the additive polynomial $h(T)$ and the set of roots of this polynomial is a $2$-dimensional vector space over $\F_p$, we conclude that $S_p$ is not abelian. It is well known that, up to isomorphism, there is a unique non-abelian group of order $p^3$ and exponent $p$, known as the Heisenberg group, see e.g. \cite[page 179]{DF}. Finally, since $\alpha \in \F_p^{*}$ and for each $\alpha$ there are precisely $p^3$ automorphisms $\sigma_{\alpha,\kappa,\epsilon}$, we conclude that $|\aut(\yy_a)|=p^3(p-1)$.
 \end{proof}

 
\section*{Acknowledgements}

The research of Nazar Arakelian was partially supported by grant 2023/03547-2, São Paulo Research Foundation (FAPESP).

\vspace{0,5cm}\noindent {\em Author's addresses}:

\vspace{0.2 cm} \noindent Nazar ARAKELIAN \\
Instituto de Ci\^encias Matem\'aticas e de Computa\c c\~ao
\\ Universidade de S\~ao Paulo \\ Avenida Trabalhador S\~ao-carlense, 400 \\
CEP 13566-590, S\~ao Carlos SP
(Brazil).\\
 E--mail: {\tt n.arakelian@icmc.usp.br}

\text{}

\vspace{0.2 cm} \noindent Diego KIAN \\
Instituto de Matem\'atica e Estat\'istica
\\ Universidade de S\~ao Paulo
 \\Rua do Mat\~ao, 1010 \\
CEP 05508-090, S\~ao Paulo SP
(Brazil).\\
 E--mail: {\tt diego.kian@alumni.usp.br}

\end{document}